\documentclass{article}
\usepackage{graphicx} 

\usepackage{setspace}

\usepackage{amsmath}
\usepackage{amssymb}
\usepackage{amsthm}
\usepackage{mathtools}
\usepackage{empheq}
\usepackage{enumitem}
\usepackage{titlesec}
\usepackage{graphicx}
\usepackage{caption}
\usepackage{subcaption}
\usepackage{diagbox}
\usepackage{hyperref}
\usepackage{xcolor}

\usepackage[T1]{fontenc}
\usepackage[utf8]{inputenc}
\usepackage{comment}
\usepackage{indentfirst}
\usepackage[top=1.25in,left=1.25in,right=1.25in]{geometry}

\numberwithin{equation}{section}

\title{\Large{\uppercase{\bf A note on generating polyhedra and quadrangulations}}}
\author{\Large{Luisa Andreis, Riccardo W. Maffucci, and Federico Polito}}
\date{}
\newcommand{\Addresses}{  
		L.~Andreis, \textsc{Dipartimento di Matematica, Universit\`a di Torino\\\indent Via Carlo Alberto 10, Turin 10123, Italy}\par\nopagebreak\vspace{-0.35cm}
		\textit{E-mail address}, L.~Andreis: \href{mailto:luisa.andreis@unito.it}{\texttt{luisa.andreis@unito.it}}
		
		R.W.~Maffucci, \textsc{Dipartimento di Matematica, Universit\`a di Torino\\\indent Via Carlo Alberto 10, Turin 10123, Italy}\par\nopagebreak\vspace{-0.35cm}
		\textit{E-mail address}, R.W.~Maffucci: \href{mailto:riccardowm@hotmail.com}{\texttt{riccardowm@hotmail.com}}
		
		F.~Polito, \textsc{Dipartimento di Matematica, Universit\`a di Torino\\\indent Via Carlo Alberto 10, Turin 10123, Italy}\par\nopagebreak\vspace{-0.35cm}
		\textit{E-mail address}, F.~Polito: \href{mailto:federico.polito@unito.it}{\texttt{federico.polito@unito.it}}
  }

\def\cp{\mathcal{P}}
\def\cq{\mathcal{Q}}
\def\calr{\mathcal{R}}
\def\ct{\mathcal{T}}

\newtheorem{thm}{Theorem}[section]

\newtheorem{prop}[thm]{Proposition}

\begin{document}
\titleformat{\section}
  {\Large\scshape}{\thesection}{1em}{}
\titleformat{\subsection}
  {\large\scshape}{\thesubsection}{1em}{}
\maketitle
\Addresses

\begin{abstract}
A polyhedron is a planar, $3$-connected graph. We iteratively construct all polyhedra (save for pyramids) from a unique starting graph, namely the square pyramid, via two graph transformations. This builds upon a previous construction, that starts from the full class of pyramids, and applies the same transformations.

In a related result, we iteratively construct all quadrangulations of the sphere where all $4$-cycles are facial, i.e., the class of radial graphs of the polyhedra (save for antibipyramids), from a unique starting graph, namely the square antibipyramid, via a unique graph transformation. This builds upon a previous construction, that starts from the full class of antibipyramids, and applies the same transformation.
\end{abstract}


\section{Introduction}
\label{sec:mainres}

In this paper, we deal with finite, undirected graphs with no loops or multiple edges. Let $S$ be a (possibly infinite) family of graphs, and $\ct$ a finite set of graph transformations. We write
\[\{S;\ct\}\]
for the class of graphs that may be generated from the graphs in $S$ by applying iteratively the transformations in $\ct$. It is commonly assumed that $S \subset \{S;\ct\}$.

The polyhedral graphs, or simply polyhedra, are defined as the wireframes ($1$-skeleta) of polyhedral solids. This is exactly the class of planar, $3$-connected graphs \cite{radste}. A graph is planar if it may be embedded in the plane (equivalently, the sphere) such that edges do not cross, except at vertices. A graph on more than $k$ vertices is $k$-connected if however one removes fewer than $k$ vertices, the graph stays connected. Clearly, by $3$-connectivity, every vertex of every polyhedron has degree at least $3$. Every polyhedron $G$ has a unique embedding in the sphere \cite{whit32}. As a consequence, its plane dual $G^*$ is unique. Moreover, $G^*$ is a polyhedron (corresponding naturally to the wireframe of the dual polyhedral solid). For an account of other relevant properties of this intriguing class of graphs, we refer the interested reader to \cite[Introduction]{maffucci2025deza}.

The $3$-connected quadrangulations of the sphere are polyhedra where every face (i.e., region) is delimited by a $4$-cycle. This is the dual class of the $4$-regular polyhedra. The smallest example of a $3$-connected quadrangulation is the cube.

We will denote by
$\cq$ the class of $3$-connected quadrangulations of the sphere where all $4$-cycles are facial (i.e., where there are no separating $4$-cycles). This sub-class of quadrangulations is extremely relevant to the study of polyhedra, as we shall now describe. Given a polyhedron $G$, we may define its radial graph $\calr(G)$ by
\[V(\calr)=V(G)\cup V(G^*)\]
and
\[E(\calr)=\{vf : v\in V(G) \text{ lies on the face corresponding to }f\in V(G^*)\}.\]
The class $\cq$ is exactly the class of radial graphs of polyhedra \cite{arcric,brin05}. That is to say, one has $R\in\cq$ if and only if there exists a pair of dual polyhedra $(G,G^*)$ such that $R\simeq\calr(G)\simeq\calr(G^*)$. The pair $(G,G^*)$ is unique, up to swapping the order.

We will denote by
\[A_n,\qquad \text{ even }n\geq 6,\]
the antibipyramid with $n$ faces (Figure \ref{fig:a}), dual of the $n/2$-gonal antiprism. This terminology comes from the fact that the bipyramids are the duals of the prisms, while the $A_n$'s are the duals of the antiprisms. One has $A_n\in\cq$ for every even $n\geq 6$. The graph $A_6$ is simply the cube. Other terms for $A_n$ that are found in the literature include `deltahedron', `trapezohedron', and `pseudo-double-wheel' \cite{brin05}. Note that, for every $n\geq 3$, one has
\[A_{2n}=\calr(W_n),\]
where $W_n$ is the $n$-gonal pyramid (i.e., wheel graph). For instance, the cube $A_6$ is the radial graph of the tetrahedron $W_3$. Another way to construct $A_n$ is to start from a cycle $C: a_1,a_2,\dots,a_n$, add a vertex $x$ and edges between $x$ and every element of odd index along $C$, and add a vertex $z$ and edges between $z$ and every element of even index along $C$.
\begin{figure}[ht]
\centering
\includegraphics[width=3cm]{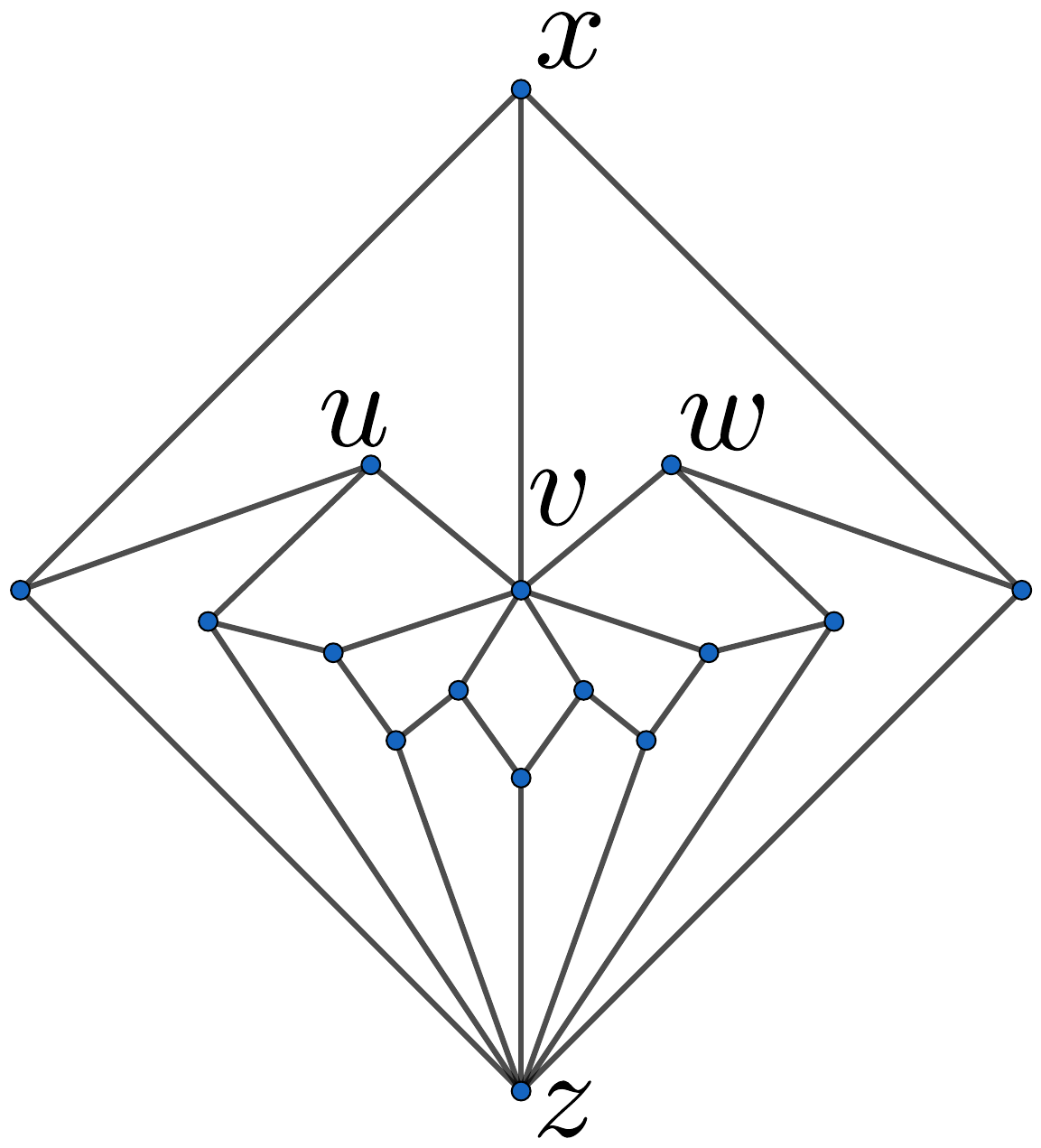}
\caption{$A_{14}$.}
\label{fig:a}
\end{figure}

The focus of this paper is the iterative construction of polyhedral graphs, and the related iterative construction of $\cq$. In 1961, Tutte established an iterative construction of all $3$-connected graphs starting from the pyramids, and using only the two transformations of edge addition and vertex expansion \cite{tutt61}. To expand a vertex $x$ in a polyhedron, we write its neighbours in cyclic order around $x$ (recalling that the plane embedding is unique for polyhedra),
\[a_1,a_2,\dots,a_m,b_1,b_2,\dots,b_n,\qquad m,n\geq 2,\]
delete the edges $xb_1,xb_2,\dots,xb_n$, and insert the new vertex $y$ and new edges $yx,yb_1,yb_2,\dots,yb_n$ (Figure \ref{fig:ex}). Clearly, vertices of degree $3$ may not be expanded. The reverse transformation to vertex expansion is called vertex contraction. As a neat consequence, Tutte established an iterative construction of all polyhedra starting from the pyramids, and using only the two transformations of edge addition and taking the dual polyhedron \cite[Theorem 6.1]{tutt61}. This is because expanding a vertex of a polyhedron is equivalent to adding an edge to the dual, and vice versa.
\begin{figure}[ht]
	\centering
	\begin{subfigure}{0.32\textwidth}
		\centering
		\includegraphics[width=2.5cm]{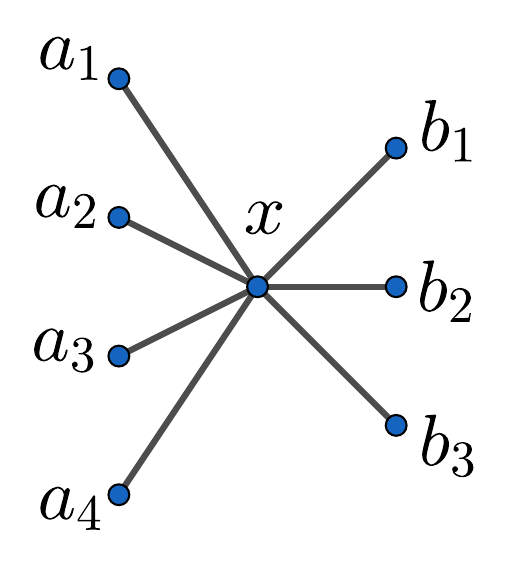}
		\caption{Before the expansion.}
		\label{fig:exa}
	\end{subfigure}
    $\rightarrow$
	\begin{subfigure}{0.32\textwidth}
		\centering
		\includegraphics[width=2.5cm]{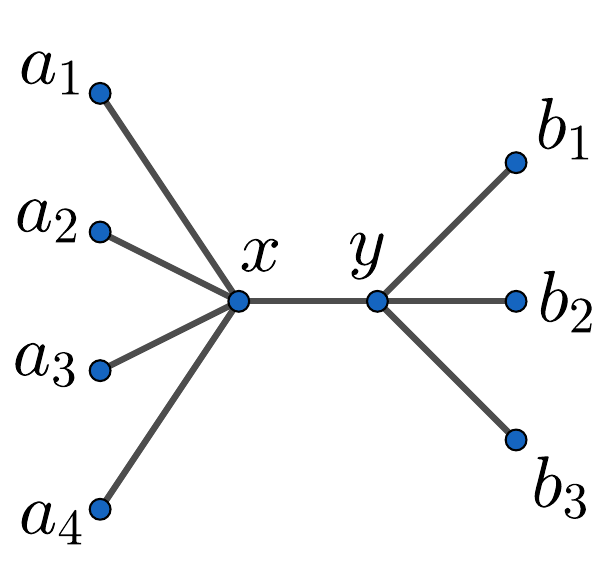}
		\caption{After the expansion.}
		\label{fig:exb}
	\end{subfigure}
    \caption{Expanding the vertex $x$.}
	\label{fig:ex}
\end{figure}

In 2005, an iterative construction of all quadrangulations of the sphere was found, as well as of a few subclasses, such as the $3$-connected quadrangulations, and $\cq$ \cite{brin05}. More precisely, it was shown in \cite[Theorem 4]{brin05} that

\[\cq=\{\{A_{n}, \text{ even }n\geq 6\};\{\cp_1\}\},\]

where $\cp_1$ is the graph transformation shown in Figure \ref{fig:p1}. For $\cp_1$ to be applicable, $v$ must be of degree at least $4$, and $u,x,w$ must be consecutive neighbours of $v$ in the cyclic order. We will also call the transformation $\cp_1$ a $\cp_1$-expansion, and the reverse operation a $\cp_1$-contraction. The $\cp_1$-expansion is not to be confused with the vertex expansion shown in Figure \ref{fig:ex}. The notation $\cp_1$ was chosen to be consistent with \cite{brin05}. As noted in \cite{brin05}, this procedure constitutes an alternative construction for all (pairs of dual) polyhedra, via the generation of their radial graphs.
\begin{figure}[ht]
	\centering
	\begin{subfigure}{0.32\textwidth}
		\centering
        \includegraphics[width=1.25cm]{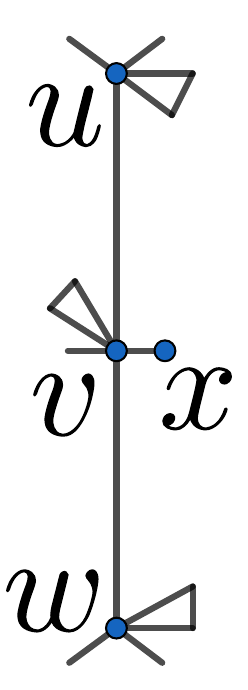}
		\caption{Before the $\cp_1$-expansion.}
		\label{fig:p1a}
	\end{subfigure}
    $\rightarrow$
	\begin{subfigure}{0.32\textwidth}
		\centering
		\includegraphics[width=3.cm]{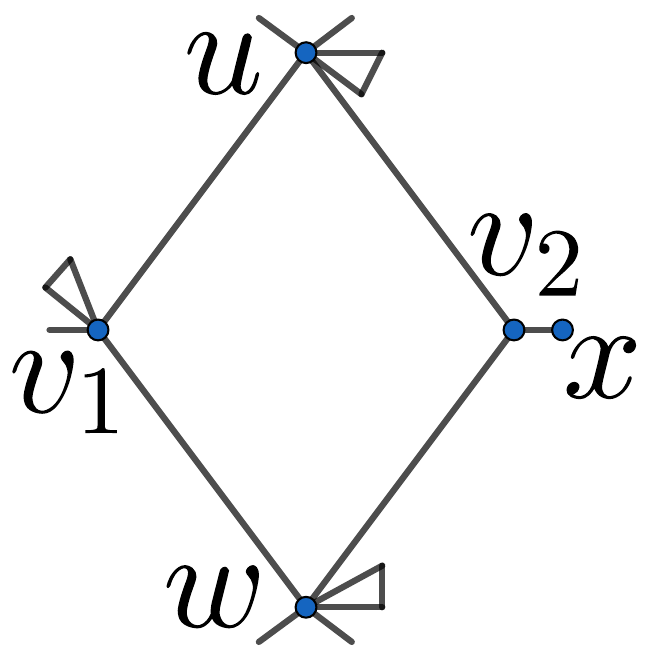}
		\caption{After the $\cp_1$-expansion.}
		\label{fig:p1b}
	\end{subfigure}
    \caption{Transformation $\cp_1$ applied to $(G,\{u,v,w,x\})$. Triangles indicate that possibly there are other edges.}
	\label{fig:p1}
\end{figure}

In our investigation, we have found a strengthening of the two mentioned results \cite[Theorem 4]{brin05} and \cite[Theorem 6.1]{tutt61}.

\begin{thm}
	\label{thm:1}
All quadrangulations of the sphere where every $4$-cycle is facial (save the antibipyramids $A_n, n \neq 8$) are generated from the antibipyramid $A_8$ via $\cp_1$-expansions.
\end{thm}

Theorem \ref{thm:1} will be proven in Section \ref{sec:pf}. It constitutes a strengthening of \cite[Theorem 4]{brin05}. One may write
\[\cq\setminus\{A_n,\ \text{even } n\neq 8\}=\{\{A_8\};\{\cp_1\}\}.\]

\begin{thm}
	\label{thm:2}
All polyhedra (save the pyramids $W_n, n \neq 4$) are generated from the square pyramid $W_4$ by adding edges and expanding vertices.
\\
Equivalently, all polyhedra (save the pyramids $W_n, n \neq 4$) are generated from the square pyramid $W_4$ by adding edges and taking duals.
\end{thm}

Theorem \ref{thm:2} will be proven in Section \ref{sec:pf}. It constitutes a strengthening of \cite[Theorem 6.1]{tutt61}.

\paragraph{Discussion.} We have managed to construct all elements of $\cq\setminus \{A_n,\ \text{even } n\neq 8\}$ from a single starting graph and a single graph transformation. This also allows us to construct all polyhedra, via their radial graphs. In the actual implementation, in terms of computational complexity and memory allocated, this does not constitute a major saving, since we are only avoiding the use of one graph at each other level. However, from a theoretical point of view, our findings greatly simplify the way we think about the iterative generation of these classes, following several decades of using the full class of pyramids/antibipyramids as starting graphs.

\section{Proofs}
\label{sec:pf}

\subsection{Proof of Theorem \ref{thm:1}}

\begin{proof}[Proof of Theorem \ref{thm:1}]
First, note the $\supset$ inclusion is trivially verified. Moreover, by \cite[Theorem 4]{brin05}, the class $\cq$ is generated from the $\{A_n, \text{ even }n\geq 6\}$ (se Figure \ref{fig:thm1a}) by applying $\cp_1$. In order to apply $\cp_1$ to $(A_n,\{u,v,w,x\})$, the vertex $v$ must be of degree at least $4$ in $A_n$, and $u,x,w$ consecutive in the cyclic order of neighbours around $v$. Hence $v$ is a vertex of degree $n/2\geq 4$. We will call $z$ the other vertex of $A_n$ of degree $n/2$. Up to isomorphism, the choice of $u,v,w,x$ is thus unique. The obtained graph is $B_{n+1}$, as in Figure \ref{fig:thm1b}. That is to say, for every even $n\geq 8$, the set of graphs generated from $A_n$ by applying $\cp_1$ once is simply $\{B_{n+1}\}$. Therefore, the present theorem will be proven if we show that
\[\{B_{n+1}, \text{ even }n\geq 8\}\subset\{\{A_8\};\{\cp_1\}\}.\]
\begin{figure}[ht]
	\centering
	\begin{subfigure}{0.32\textwidth}
		\centering
		\includegraphics[width=4cm]{a14.pdf}
		\caption{$A_{14}$.}
		\label{fig:thm1a}
	\end{subfigure}
	\begin{subfigure}{0.32\textwidth}
		\centering
		\includegraphics[width=4cm]{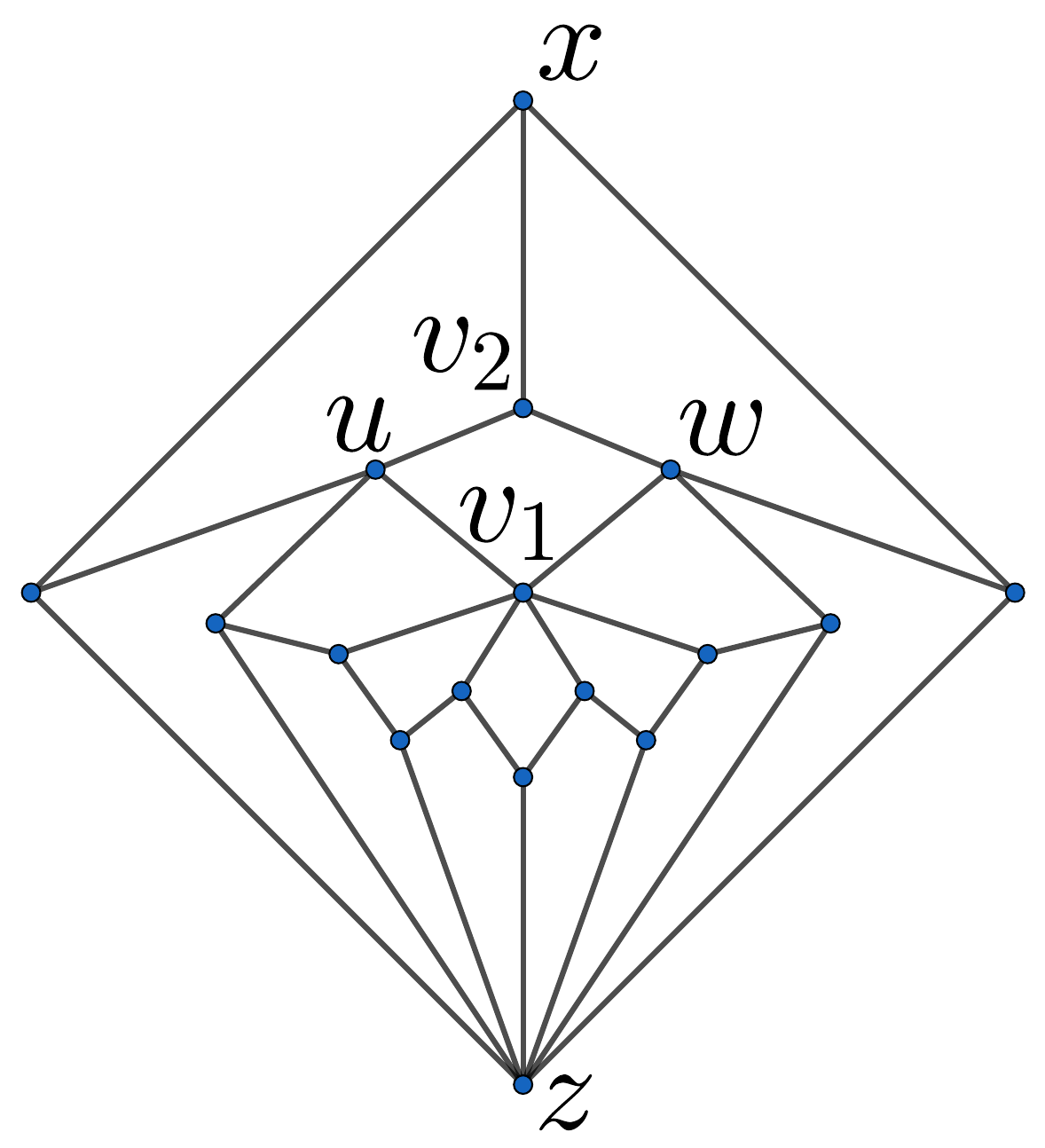}
		\caption{$B_{15}$.}
		\label{fig:thm1b}
	\end{subfigure}
	\begin{subfigure}{0.32\textwidth}
		\centering
		\includegraphics[width=3.5cm]{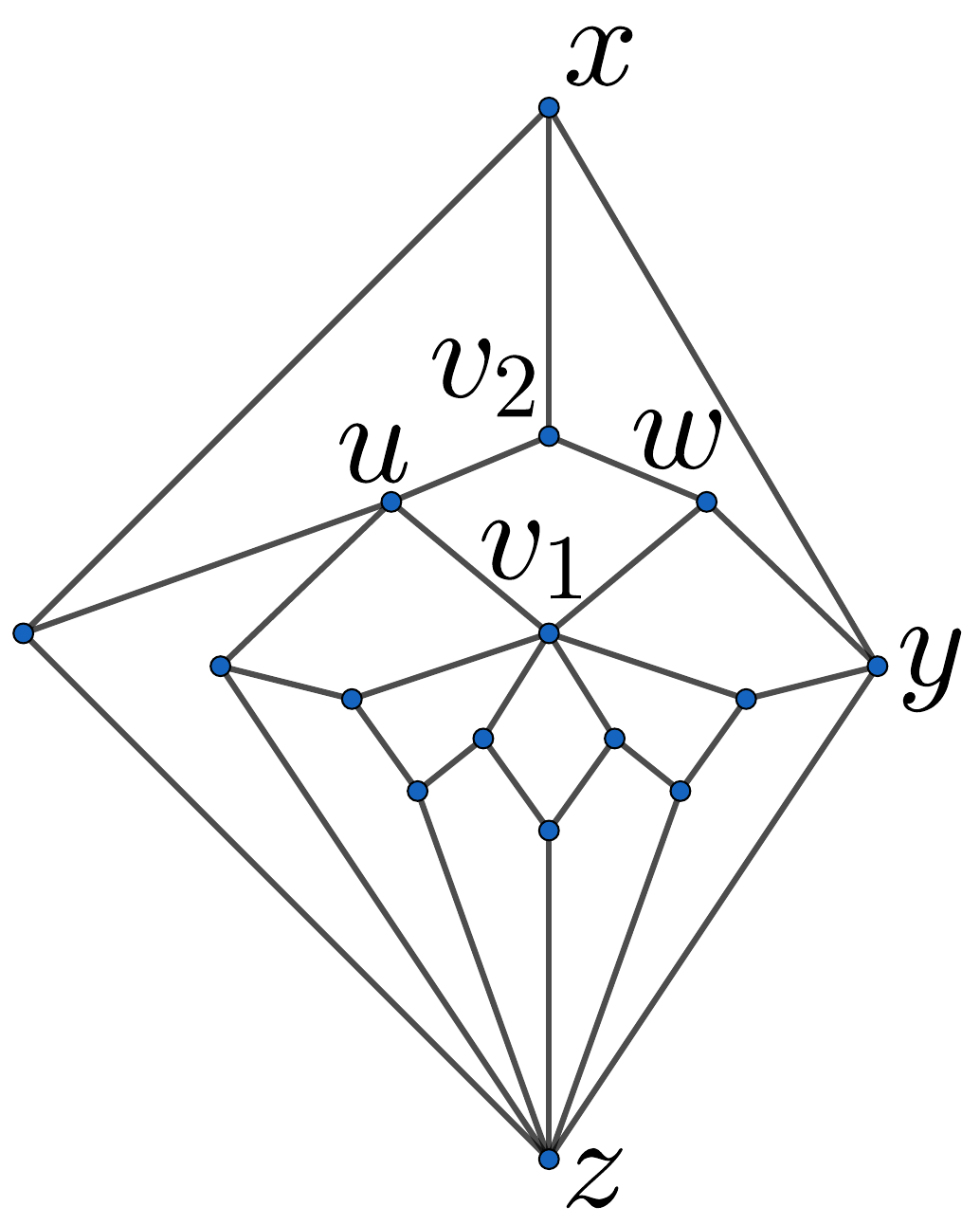}
		\caption{$C_{14}$.}
		\label{fig:thm1c}
	\end{subfigure}
	\caption{Proof of Theorem \ref{thm:1}.}
	\label{fig:thm1}
\end{figure}

To apply $\cp_1^{-1}$, we need a face with two diagonally opposite vertices of degree at least $4$ (in Figure \ref{fig:p1b}, they are $u,w$). In $B_{n+1}$, the only vertices of degree at least $4$ are: $z$, of degree $n/2$; $v_1$, of degree $n/2-1$; $u,w$, of degree $4$. Hence, up to isomorphism, there are only two faces that one may contract to apply $\cp_1^{-1}$ to $B_{n+1}$: $[u,v_1,w,v_2]$, that yields $A_n$, and the face containing $w,z$ (or equivalently up to isomorphism, $u,z$), that yields the graph $C_{n}$ illustrated in Figure \ref{fig:thm1c}.

Now in $C_{n}$, the vertices of degree at least $4$ are: $z,v_1$, of degree $n/2-1$; $u,y$, of degree $4$. Hence $\cp_1^{-1}$ is applicable to either the face containing $v_1,y$, or the one containing $u,z$. Either choice yields $B_{n-1}$.

We have proven that one may contract $B_{n+1}$ to $C_n$ and then to $B_{n-1}$, via two applications of $\cp_1^{-1}$. We continue in this fashion, obtaining successively
\[B_{n+1},C_n,B_{n-1},C_{n-2},\dots,B_9,C_8\simeq A_8.\]
\end{proof}

In the above proof, we have also shown that, to contract $B_{n+1}$, even $n\geq 8$, to an antibipyramid, the only way is given by the successive contractions
\[B_{n+1},C_n,B_{n-1},C_{n-2},\dots,B_{m+1},C_m\simeq A_m,\qquad \text{ even }m\geq 8.\]

\subsection{Further insight on Theorem \ref{thm:1}}

In the proof of Theorem \ref{thm:1}, another way to see that indeed at each step the $\cp_1$-contraction yields the sought $B_n$ or $C_n$, is to use degree sequences. For any graph, the degree sequence lists the degrees of its vertices in non-increasing order. For instance, the antibipyramid $A_{2n}$ has sequence
\begin{equation}
\label{eq:seq}
n,n,3^{2n},
\end{equation}
with the notation $3^{2n}$ indicating that the value $3$ is repeated $2n$ times. We say that $A_{2n}$ realises the degree sequence \eqref{eq:seq}.

To check that contracting $B_{n+1}$ yields $C_n$, and contracting $C_{n}$ yields $B_{n-1}$, firstly one proves (as we shall do next) that the $B_n$'s and $C_n$'s (with the exception of $C_{12}$) are {\bf unigraphic} in $\cq$, in the sense that each is the unique element of $\cq$ realising its degree sequence. Next, one combines this with the fact that, as we apply a $\cp_1$-contraction to $B_n$ or $C_n$, we know immediately the degree sequence of the resulting graph. For the theory of unigraphic polyhedra in general, see \cite{maffucci2025faces} and the references therein.

\begin{prop}
For every $n\geq 3$, $A_{2n}$ is the unique element of $\cq$ with degree sequence
\[n,n,3^{2n}.\]
For every $n\geq 4$, $B_{2n+1}$ is the unique element of $\cq$ with degree sequence
\[n,n-1,4,4,3^{2n-1}.\]
For $n=4$ and for every $n\geq 6$, $C_{2n+2}$ is the unique element of $\cq$ with degree sequence
\[n,n,4,4,3^{2n}.\]
Moreover, for every $n\geq 4$, $C_{2n+2}$ is the radial graph of the self-dual polyhedron $S(n,4)$ that appeared in \cite{maffucci2026self}.
\end{prop}
\begin{proof}
We will prove the statements for $C_{2n+2}$, the other proofs being of similar flavour. Let $R\in\cq$ have degree sequence
\[n,n,4,4,3^{2n}.\]
By the handshaking lemma,
\[|E(R)|=4n+4,\]
and since $R$ is a quadrangulation, we deduce that $R$ has $2n+2$ faces. For $n=4$, we note that, apart from the antibipyramid $A_{10}$, the graph $C_{10}$ is the only element of $\cq$ with $10$ faces, so that the statement of the present proposition is clearly true for $n=4$.

Henceforth, we thus take $n\geq 5$. We write $R=\calr(G)=\calr(G^*)$, where $(G,G^*)$ is a pair of dual polyhedra. 

Let $G$ have two faces of length $n$. Two distinct faces in a polyhedron may have at most two common vertices \cite[Section 1]{maffucci2025deza}, hence
\[|V(G)|\geq n+n-2.\]
We combine the above with $3|V(G)|\leq 2|E(G)|$, that holds by $3$-connectivity of $G$, and obtain

\[6n-6\le 3|V(G)|\leq 2|E(G)|=4n+4,\]

whence $n\leq 5$. We inspect the elements of $\cq$ with $12$ faces to conclude that exactly two of them have sequence $5,5,4,4,3^{10}$.

Henceforth, we take $n\geq 6$. As argued above, $G,G^*$ each have exactly one vertex of degree $n$ and one face of length $n$. Now suppose by contradition that $G$ has two vertices of degree $4$. It follows that $G$ has degree sequence
\[n,4,4,3^x\]
for some integer $x$ satisfying $0\leq x\leq 2n$. Since $R=\calr(G)=\calr(G^*)$ and $|E(R)|=4n+4$, one has
\[|E(G)|=|E(G^*)|=2n+2.\]
By the handshaking lemma for $G$,
\[n+4+4+3x=2(2n+2),\]
whence $x=n-4/3$, a contradiction.

We deduce that both $G,G^*$ have degree sequence
\[n,4,3^{n}.\]
The $n$-gonal face
\[\alpha_n=[u_1,u_2,\dots,u_n]\]
of $G$ cannot be adjacent only to triangular faces, else either $G$ would be the $n$-gonal pyramid, or $G$ would have at least three vertices of degree at least $4$. Hence in $G$, $\alpha_n$ and the quadrangular face $\alpha_4$ are adjacent. We may write
\[\alpha_4=[u_1,u_2,v,w].\]
Together, these two faces have $n+2$ distinct vertices, so that
\[V(G)=\{u_1,u_2,\dots,u_n,v,w\}.\]
Moreover, for $i\in\{4,n\}$, each vertex on $\alpha_i$ is adjacent to exactly two other vertices on $\alpha_i$, because $G$ is a polyhedron \cite[Section 1]{maffucci2025deza}. On the other hand, there exists in $G$ a vertex of degree $n$. It follows that this vertex is one of $v,w$, say $v$, and it is adjacent to all of
\[u_3,u_4,\dots,u_n.\]
Now $w$ has exactly one neighbour among $u_3,u_4,\dots,u_n$. By planarity, the only possibility is $wu_n\in E(G)$. For every $n\geq 6$, $G$ has been uniquely determined.

For every $n\geq 4$, we note that $C_{2n+2}=\calr(S(n,4))$, where $S(n,4)$ was defined in \cite{maffucci2026self}. It is the only self-dual polyhedron of degree sequence $n,4,3^{n}$ \cite[Theorem 1.3]{maffucci2026self}.
\end{proof}

\subsection{Proof of Theorem \ref{thm:2}}

\begin{proof}[Proof of Theorem \ref{thm:2}]
An equivalent way of generating the polyhedra from the pyramids is to generate their radials, i.e., the members of $\cq$, from the antibipyramids. The present theorem now follows from Theorem \ref{thm:1}.

An alternative way to prove Theorem \ref{thm:2} is the following. By \cite[Theorem 6.1]{tutt61}, every polyhedron may be constructed from a pyramid by adding edges and expanding vertices as in Figure \ref{fig:ex}. Since pyramids are self-dual, the first step in this procedure is always to add an edge. We write $u$ for the apex of the $n$-gonal pyramid $W_n$, and
\[a_1,a_2,\dots,a_n\]
for the base. Then the graph obtained from $W_n$ by adding the edge $a_1a_{m+2}$ may be denoted by
\[W(n,m), \qquad 1\leq m\leq n/2-1\]
(note that it suffices to consider $m\leq n/2-1$ to generate all possibilities up to isomorphism). The present theorem will thus be proven if we can generate each $W(n,m)$, $n\geq 4$, $1\leq m\leq n/2-1$ starting from $W_4$ via edge addition and vertex expansion.

Given $W(n,m)$, we delete $ua_1$. Then, we may contract: either the vertices $a_1,a_2$, obtaining a graph isomorphic to $W(n-1,m)$; or the vertices $a_1,a_n$, obtaining a graph isomorphic to $W(n-1,m-1)$. Proceeding iteratively via these two steps of deletion and contractions, we obtain successively
\[W(n,m),W(n-1,m-1),\dots,W(n-m+1,1),W(n-m,1),\dots,W(4,1).\]
Finally, $W_4$ is obtained from $W(4,1)$ by deleting an edge connecting two vertices of degree $4$.
\end{proof}

\paragraph{Acknowledgements.}
Riccardo W. Maffucci was partially supported by Programme for Young Researchers `Rita Levi Montalcini' PGR21DPCWZ \textit{Discrete and Probabilistic Methods in Mathematics with Applications}, awarded to Riccardo W. Maffucci.

\bibliographystyle{abbrv}
\bibliography{bibgra}
\end{document}